\documentclass[graybox,envcountsame]{svmult}
\usepackage{amsmath,amssymb,enumerate}
\usepackage[final]{hyperref}

\newcommand{\bb}[1]{\mathbb{#1}}
\newcommand{\ol}{\overline}
\newcommand{\bs}{\backslash}

\newcommand{\pd}{\partial}
\newcommand{\Om}{\Omega}
\newcommand{\om}{\omega}
\newcommand{\W}{{\cal W}}
\newcommand{\PI}{{\cal{PI}}}
\newcommand{\supp}{\operatorname{supp}}
\newcommand{\CAP}{\operatorname{cap}}
\renewcommand{\Re}{\operatorname{Re}}

\newcommand{\sgn}{\operatorname{sgn}}

\newcommand{\eps}{\varepsilon}
\newcommand{\norm}[1]{\|#1\|}

\begin{document}


\title*{Weighted Chebyshev Polynomials on Compact Subsets of the Complex Plane}
\author{Galen Novello, Klaus Schiefermayr, and Maxim Zinchenko}
\institute{Galen Novello \at Department of Mathematics and Statistics, University of New Mexico, Albuquerque, NM 87131, USA; \email{gnovello@unm.edu}
\and Klaus Schiefermayr \at University of Applied Sciences Upper Austria, Campus Wels, Austria; \email{klaus.schiefermayr@fh-wels.at}
\and Maxim Zinchenko \at Department of Mathematics and Statistics, University of New Mexico, Albuquerque, NM 87131, USA; \email{maxim@math.unm.edu};
Research supported in part by Simons Foundation grant CGM--581256.}
\maketitle

\abstract*{}

\abstract{
We study weighted Chebyshev polynomials on compact subsets of the complex plane with respect to a bounded weight function. We establish existence and uniqueness of weighted Chebyshev polynomials and derive weighted analogs of Kolmogorov's criterion, the alternation theorem, and a characterization due to Rivlin and Shapiro. We derive invariance of the Widom factors of weighted Chebyshev polynomials under polynomial pre-images and a comparison result for the norms of Chebyshev polynomials corresponding to different weights. Finally, we obtain a lower bound for the Widom factors in terms of the Szeg\H{o} integral of the weight function and discuss its sharpness.
}

\bigskip

\noindent\emph{Mathematics Subject Classification (2020): 41A50, 30C10, 30E10}

\noindent\emph{Keywords: Weighted Chebyshev polynomials, Bernstein--Walsh inequality, Szeg\H{o} lower bound}

\section{Introduction}


Let $K\subset\bb C$ be a compact set in the complex plane. In this paper, we consider \emph{weighted Chebyshev polynomials} with respect to a weight function $w:K\to[0,\infty)$. For this purpose, we denote by $T_{n,w}^{(K)}(z)$ the weighted Chebyshev polynomial of degree~$n$ on $K$ with respect to~$w$, that is, the minimizer of $\norm{P_n}_{K,w}:=\norm{wP_n}_{K}$ among all monic polynomials~$P_n$ of degree~$n$, where $\norm{\cdot}_{K}$ is the supremum norm on $K$. Moreover, we are interested in the norm of $T_{n,w}^{(K)}$, denoted by
\begin{equation}
t_n(K,w):=\norm{T_{n,w}^{(K)}}_{K,w}.
\end{equation}
To avoid the trivial case, we assume that $w$ is nonzero on at least $n+1$ points of $K$ in which case $\|\cdot\|_{K,w}$ defines a norm on the space of degree $n$ polynomials and hence $t_n(K,w)>0$.
When $K$ has positive logarithmic capacity $\CAP(K)$ we are also interested in the so-called \emph{Widom factors}, defined by
\begin{equation}\label{Wfactors}
\W_n(K,w):=\frac{t_n(K,w)}{\CAP(K)^n}.
\end{equation}


Our contribution in this paper is twofold. First, in Section~\ref{S2}, we discuss existence, uniqueness, and characterization of weighted Chebyshev polynomials for rather general weight functions $w$ on a compact set $K\subset\bb C$. In particular, we derive weighted analogs of Kolmogorov's criterion, the Rivlin--Shapiro characterization, and the alternation theorem in the case of $K\subset\bb R$. Our discussion parallels the classical case of the constant weight $w_0=1$. While the topic of weighted Chebyshev polynomials is certainly well known to experts we were not able to find a suitable reference that addresses the basic questions in full generality. As an application of Kolmogorov's criterion we also derive an invariance result for the weighted Widom factors under polynomial pre-images.

Our second contribution consists of several inequalities for norms and Widom factors of weighted Chebyshev polynomials. In Section~\ref{S3} we derive a comparison result for Chebyshev norms corresponding to two different weights. Applied to the case of a constant weight the comparison result provides a way of estimating weighted Chebyshev norms in terms of the unweighted ones. Then we discuss a class of Szeg\H{o} weights and derive a weighted analog of the Bernstein--Walsh inequality. We finish the paper with a detailed discussion of a Szeg\H{o}-type lower bound for the weighted Widom factors,
\begin{align}
\W_n(K,w) \ge \exp\left[\int\log w(\zeta)\,d\rho_K(\zeta)\right],
\end{align}
where $\rho_K$ denotes the equilibrium measure of $K$. We derive a characterization for equality in the above lower bound for a given $n$ and as a consequence show that the inequality is always strict when $K\subset\bb R$. In the real setting we also show that unlike the unweighted case the above weighted lower bound is sharp in the class of polynomial weights $w$.

\section{Existence, Uniqueness, and Characterization of Weighted Chebyshev Polynomials}\label{S2}

In this section we discuss existence, uniqueness, and characterization of weighted Chebyshev polynomials. Throughout the section we assume that $K\subset\bb C$ is a compact set and $w:K\to[0,\infty)$ is a bounded weight function nonzero on at least $n+1$ points of $K$. We note that $w$ is allowed to have discontinuities and zeros on $K$. The above assumptions on $w$ are the minimal ones that guarantee that $\|\cdot\|_{K,w}$ is a norm on the space of polynomials of degree $n$.
However, for our discussion of weighted Chebyshev polynomials it will be convenient to assume that $w$ is upper semi-continuous (i.e., $w(z)\ge\limsup_{K\ni\zeta\to z}w(\zeta)$ for each accumulation point $z$ of $K$). The following lemma shows that this can be done without loss of generality by replacing $w$ with $\hat w:K\to[0,\infty)$ defined by
\begin{align}\label{w-usc}
\hat w(z)&=\lim_{r\to0^+}\sup_{\zeta\in D(z,r)\cap K}w(\zeta), \quad z\in K,
\end{align}
where $D(z,r)$ denotes the open disk of radius $r$ centered at $z$. 

\begin{lemma}\label{USC-Lem}
The weight function $\hat w$ is upper semi-continuous on $K$, $\hat w\ge w$ on $K$, and $\norm{wf}_{K}=\norm{\hat wf}_{K}$ for any continuous function $f$ on $K$.
\end{lemma}
\begin{proof}
Let $z_0\in K$ and $\eps>0$. Then, by \eqref{w-usc}, there exists $r>0$ such that
$w(z)\le \hat w(z_0)+\eps$ for all $z\in D(z_0,r)\cap K$. This together with \eqref{w-usc} implies that $\hat w(z)\le \hat w(z_0)+\eps$ for all $z\in D(z_0,r)\cap K$. Thus, $\hat w$ is upper semi-continuous on $K$. We also get the inequality $w(z_0)\le\hat w(z_0)$.

Let $f$ be a continuous function on $K$. Since $\hat w$ is upper semi-continuous and $K$ is compact, there exists $z_0\in K$ such that $\hat w(z_0)|f(z_0)|=\norm{\hat w f}_{K}$. Using \eqref{w-usc}, pick a sequence $\{z_n\}\subset K$ such that $z_n\to z_0$ and $w(z_n)\to \hat w(z_0)$. Then $\norm{wf}_{K}\ge \hat w(z_0)|f(z_0)|=\norm{\hat w f}_{K}$. The opposite inequality $\norm{wf}_{K}\le\norm{\hat wf}_{K}$ follows from $w\le\hat w$.
\end{proof}

Let $\bb P_n$ denote the space of polynomials of degree at most $n$. By assumption $w(z)\neq0$ for at least $n+1$ points $z\in K$, hence $\norm{\cdot}_{K,w}$ is a norm on $\bb P_n$. Since every finite dimensional normed linear space over $\bb C$ is complete, there exists $p_{n-1}\in\bb P_{n-1}$ that minimizes $\norm{z^n-p_{n-1}}_{K,w}$. Thus, there always exists a weighted Chebyshev polynomial on $K$ given by $T_{n,w}^{(K)}=z^n-p_{n-1}$.

Next, we show that Kolmogorov's criterion holds for weighted Chebyshev polynomials. We call a point $z\in K$ a \emph{$w$-extremal point} of a polynomial $P$ on $K$ if $w(z)|P(z)|=\norm{wP}_{K}=\norm{P}_{K,w}$. In general a polynomial $P$ might have no $w$-extremal points on $K$ due to discontinuities of $w$, however the next lemma shows that such pathological cases do not occur when $w$ is upper semi-continuous.

\begin{lemma}\label{ExtrPts-Lem}
If $w$ is an upper semi-continuous weight on $K$, then for any polynomial $P$ the set of all its $w$-extremal points $K_0$ is a nonempty compact set. If $\norm{P}_{K,w}>0$ the restriction of $w$ to $K_0$ is a positive continuous function.
\end{lemma}
\begin{proof}
Since $w$ is upper semi-continuous, so is $w|P|$ and hence $w|P|$ attains a maximum on the compact set $K$. Thus, the set $K_0\subset K$ of $w$-extremal points is nonempty. Since $w|P|\le\norm{wP}_{K}$ with equality attained on $K_0$, it follows from the upper semi-continuity of $w|P|$ that any limit point of $K_0$ is a $w$-extremal point so $K_0$ is closed and hence compact. If $\norm{wP}_{K}>0$ we have $|P|>0$ on $K_0$ hence $1/|P|$ is continuous on $K_0$ and so the restriction $w|_{K_0}=\norm{wP}_{K}/|P|$ is a positive continuous function.
\end{proof}


\begin{theorem}\label{Kolm-Thm}
Suppose $w$ is an upper semi-continuous weight on $K$ and $P_n$ is a monic polynomial of degree $n$. Then $P_n$ is a weighted Chebyshev polynomial on $K$ if and only if for each $q\in\bb P_{n-1}$,
\begin{align}\label{Kolm-Crit}
\max_{z_0\in K_0}\Re\bigr[P_n(z_0)\ol{q(z_0)}\,\bigl]\ge0,
\end{align}
where $K_0$ is the compact set of $w$-extremal points of $P_n$ on $K$.
\end{theorem}
\begin{remark}
In \eqref{Kolm-Crit} one can replace $P_n(z_0)$ with $\sgn(P_n(z_0))$, where the sign function is defined by $\sgn(z)=z/|z|$.
\end{remark}
\begin{proof}
Suppose \eqref{Kolm-Crit} holds and let $z_0\in K_0$ denote a point for which the maximum is attained. Then for any monic polynomial $Q$ of degree $n$ we have $Q=P_n+q$ for some $q\in\bb P_{n-1}$ and so, by \eqref{Kolm-Crit},
\begin{align}
|Q(z_0)|^2=|P_n(z_0)|^2+|q(z_0)|^2+2\Re\bigl[P_n(z_0)\ol{q(z_0)}\,\bigr]\ge|P_n(z_0)|^2.
\end{align}
It follows that
\begin{align}
\norm{Q}_{K,w}\ge|w(z_0)Q(z_0)|\ge|w(z_0)P_n(z_0)|=\norm{P_n}_{K,w},
\end{align}
so $P_n$ is a weighted Chebyshev polynomial on $K$.

Conversely, suppose \eqref{Kolm-Crit} does not hold, that is, there exists $q\in\bb P_{n-1}$ such that $\Re[P_n(z_0)\ol{q(z_0)}]<0$ for all $z_0\in K_0$. Then since $K_0$ is compact, there exists an open set $U\supset K_0$ and $M>0$ such that $\Re[P_n(z)\ol{q(z)}]\le-M$ for all $z\in\ol{U}$. Hence for all sufficiently small $\eps>0$ the $n$-th degree monic polynomial $Q_\eps=P_n+\eps q$ satisfies
\begin{align}\label{QleP}
|Q_\eps(z)|^2&=|P_n(z)|^2+\eps^2|q(z)|^2+2\eps\Re\bigl[P_n(z)\ol{q(z)}\,\bigr]\nonumber\\
&\le|P_n(z)|^2+\eps^2\norm{q}_K^2-2\eps M<|P_n(z)|^2, \quad z\in\ol{U}\cap K.
\end{align}
Since $w$ is upper semi-continuous, so is $w|Q_\eps|$ and hence $w|Q_\eps|$ attains a maximum on the compact set $\ol{U}\cap K$, that is, $\norm{Q_\eps}_{\ol{U}\cap K,w}=w(z_1)|Q_\eps(z_1)|$ for some $z_1\in\ol{U}\cap K$. Then, by \eqref{QleP}, $\norm{Q_\eps}_{\ol{U}\cap K,w}<w(z_1)|P_n(z_1)|\le\norm{P_n}_{K,w}$ if $w(z_1)\neq0$ and otherwise $\norm{Q_\eps}_{\ol{U}\cap K,w}=0<\norm{P_n}_{K,w}$. In either case, $\norm{Q_\eps}_{\ol{U}\cap K,w}<\norm{P_n}_{K,w}$.

Now we estimate $\norm{Q_\eps}_{K\bs U,w}$. Since $K\bs U$ is compact and $w|P_n|$ is upper semi-continuous, there exists $z_2\in K\bs U$ such that $\norm{P_n}_{K\bs U,w}=w(z_2)|P_n(z_2)|$. Since $K_0\subset U$, $z_2\notin K_0$ and hence $\norm{P_n}_{K\bs U,w}=w(z_2)|P_n(z_2)|<\norm{P_n}_{K,w}$. Then $\norm{Q_\eps}_{K\bs U,w}\le\norm{P_n}_{K\bs U,w}+\eps\norm{q}_{K\bs U,w}<\norm{P_n}_{K,w}$ for sufficiently small $\eps>0$. Combining with the earlier inequality $\norm{Q_\eps}_{\ol{U}\cap K,w}<\norm{P_n}_{K,w}$ we get $\norm{Q_\eps}_{K,w}<\norm{P_n}_{K,w}$, so $P_n$ is not a weighted Chebyshev polynomial on $K$.
\end{proof}

\begin{corollary}
If $w$ is upper semi-continuous on $K$, then a weighted Chebyshev polynomial of degree $n$ on $K$ has at least $n+1$ $w$-extremal points.
\end{corollary}
\begin{proof}
Suppose by contradiction that a weighted Chebyshev polynomial $T_{n,w}^{(K)}$ has $m\le n$ $w$-extremal points $z_1,\dots,z_m$ on $K$. By Lagrange interpolation there exists $q\in\bb P_{n-1}$ such that $q(z_j)=-T_{n,w}^{(K)}(z_j)$ for all $j=1,\dots,m$. Since $\norm{T_{n,w}^{(K)}}_{K,w}>0$ we have $T_{n,w}^{(K)}(z_j)\neq0$, and hence,
\begin{align}
\Re\bigr[T_{n,w}^{(K)}(z_j)\ol{q(z_j)}\,\bigl]=-|T_{n,w}^{(K)}(z_j)|^2<0, \quad j=1,\dots,m,
\end{align}
which is a contradiction to Kolmogorov's criterion.
\end{proof}

\begin{corollary}\label{Cor-Uniq}
The weighted Chebyshev polynomial $T_{n,w}^{(K)}$ is unique.
\end{corollary}
\begin{proof}
By Lemma~\ref{USC-Lem} we may assume without loss of generality that $w$ is upper semi-continuous on $K$. Let $P_n$ be a weighted Chebyshev polynomials of degree $n$ on $K$, then $t_n(K,w)=\norm{T_{n,w}^{(K)}}_{K,w}=\norm{P_n}_{K,w}$.

Let $Q=\frac12(P_n+T_{n,w}^{(K)})$, then $Q$ is a monic polynomial of degree $n$ and, by triangle inequality, $\norm{Q}_{K,w}\le t_n(K,w)$. Thus, $Q$ is also a weighted Chebyshev polynomial on $K$ and so, by the previous corollary, $Q$ has at least $n+1$ $w$-extremal points $z_1,\dots,z_{n+1}$. Then, by triangle inequality, we get for each $j=1,\dots,n+1$,
\begin{align}
\frac{t_n(K,w)}{w(z_j)}=|Q(z_j)|&=\frac12\big|P_n(z_j)+T_{n,w}^{(K)}(z_j)\big| \nonumber\\
&\le\frac12\Big(\big|P_n(z_j)\big|+\big|T_{n,w}^{(K)}(z_j)\big|\Big)=\frac{t_n(K,w)}{w(z_j)},
\end{align}
which implies $P_n(z_j)=T_{n,w}^{(K)}(z_j)$, $j=1,\dots,n+1$, and hence, $P_n\equiv T_{n,w}^{(K)}$.
\end{proof}


In the special case $K\subset\bb R$ the weighted Kolmogorov criterion yields a weighted version of the well known Alternation Theorem:


\begin{theorem}
Suppose $K\subset\bb R$ and let $w$ be an upper semi-continuous weight on $K$. Then:
\begin{enumerate}
\item The $w$-Chebyshev polynomial $T_{n,w}^{(K)}$ is real.
\item Let $P_n$ be a real monic polynomial of degree $n$. Then $P_n=T_{n,w}^{(K)}$ if and only if $P_n$ has $n$ sign changes on the set of its $w$-extremal points on $K$, that is, there exist points $x_0<x_1<\dots<x_n$ on $K$ such that
\begin{equation}\label{alt-signs}
w(x_j)P_n(x_j)=(-1)^{n-j}\norm{P_n}_{K,w}, \quad j=0,1,\ldots,n.
\end{equation}
\end{enumerate}
\end{theorem}
\begin{proof}
(i) Since $|\!\Re[T_{n,w}^{(K)}(x)]|\le|T_{n,w}^{(K)}(x)|$ and $\Re[T_{n,w}^{(K)}(x)]$ is a monic polynomial of $x\in\bb R$, it follows from uniqueness of the weighted Chebyshev polynomial, Corollary~\ref{Cor-Uniq}, that $T_{n,w}^{(K)}$ must be real.\\
(ii) Let $P_n$ be a real monic polynomial of degree $n$ such that $P_n$ alternates sign on some $w$-extremal points $x_0<x_1<\dots<x_n$ and suppose by contradiction that $P_n$ is not the $w$-Chebyshev polynomial. Then by Kolmogorov's criterion, Theorem~\ref{Kolm-Thm}, there exists a polynomial $q\in\bb{P}_{n-1}$ such that
\begin{align}
\Re\bigr[P_n(x_j)\ol{q(x_j)}\,\bigl]<0,\quad j=0,\dots,n.
\end{align}
Then, the real polynomial $r(x)=\Re[q(x)]$ of degree at most $n-1$ alternates sign on $x_0<x_1<\dots<x_n$ hence $r(x)$ has at least $n$ zeros and so $r(x)$ must be identically zero, a contradiction.

Now, let $P_n=T_{n,w}^{(K)}$ and suppose by contradiction that $P_n$ has fewer than $n$ sign changes on the set of its $w$-extremal points. Partition the real line by the real zeros of $P_n$ into $m\le n+1$ sets (open intervals). If $P_n$ has the same sign on two consecutive sets, replace the two sets by their union and repeat the process until $P_n$ takes alternating signs on consecutive sets. Then if a set does not contain any $w$-extremal point replace the set and its nearest neighbors by their union. This will result in $m\le n$ sets since by assumption $P_n$ has at most $n-1$ sign changes on its $w$-extremal points. Now let $q$ be the degree $m-1$ polynomial with zeros between consecutive sets and leading coefficient $-1$. Then $q\in\bb{P}_{n-1}$ and by construction $q$ has the opposite sign to $P_n$ at each $w$-extremal point of $P_n$. The existence of such a polynomial $q$ contradicts \eqref{Kolm-Crit}. Thus, $P_n$ must have $n$ real simple zeros which partition $\bb R$ into $n+1$ intervals each of which contains a $w$-extremal point and hence $P_n$ has $n$ sign changes on the set of its $w$-extremal points.
\end{proof}

We also mention a reformulation of Kolmogorov's criterion due to Rivlin and Shapiro.
The following weighted version follows from Theorem~\ref{Kolm-Thm} as in the classical unweighted case \cite[Thm.\,2.4]{DVoLo-1993}.


\begin{theorem}\label{Rivl-Thm}
Suppose $w$ is an upper semi-continuous weight on $K$ and $P_n$ is a monic polynomial of degree $n$. Then $P_n$ is a weighted Chebyshev polynomial on $K$ if and only if there exist $m\le2n+1$ $w$-extremal points $z_1,\dots,z_m$ and positive number $\alpha_1,\dots,\alpha_m$ such that $\sum_{j=1}^m\alpha_j=1$ and for each $q\in\bb P_{n-1}$,
\begin{align}\label{Rivl-Crit1}
\sum_{j=1}^m\alpha_jP_n(z_j)\ol{q(z_j)}=0.
\end{align}
\end{theorem}

\begin{corollary}
Suppose $w$ is an upper semi-continuous weight on $K$ and $P_n$ is a monic polynomial of degree $n$. Then $P_n$ is the weighted Chebyshev polynomial on $K$ if and only if there exist $n+1\le m\le2n+1$ $w$-extremal points $z_1,\dots,z_m$ and positive numbers $\lambda_1,\dots,\lambda_m$ such that
\begin{align}\label{Rivl-Crit2}
\sum_{j=1}^m\lambda_jz_j^k\sgn\ol{P_n(z_j)}=0, \quad k=0,\dots,n-1,
\end{align}
where $\sgn(z)=z/|z|$.
\end{corollary}
\begin{proof}
This is a straightforward restatement of the previous theorem. First, note that \eqref{Rivl-Crit1} may hold only if $m>n$ since otherwise, by Lagrange interpolation, there exists $q\in\bb P_{n-1}$ such that $q(z_j)=P_n(z_j)$ for all $j=1,\dots,m$ making the left hand-side of \eqref{Rivl-Crit1} strictly positive.

By linearity, it suffices to consider in \eqref{Rivl-Crit1} only the monomials $q(z)=z^k$, $k=0,\dots,n-1$. In addition, since \eqref{Rivl-Crit1} is unaffected by an overall factor, the normalization $\sum_{j=1}^m\alpha_j=1$ can be dropped. Then equivalence of \eqref{Rivl-Crit1} and \eqref{Rivl-Crit2} follows from setting $\lambda_j=\alpha_j|P_n(z_j)|$ and taking complex conjugation.
\end{proof}


We finish this section with a weighted analog of the invariance of Chebyshev norms and Widom factors under polynomial pre-images \cite{KamoBorodin-1994}, \cite[Sect.\,6]{CSZ-2020}.

\begin{theorem}
Let $L\subset\bb C$ be a compact set containing at least $n+1$ points, $w_L$ an upper semi-continuous weight on $L$, and $p(z)=a_m z^m+\dots+a_0$ a polynomial of degree $m$. Define $K=p^{-1}(L)$ and $w_K(z)=w_L(p(z))$. Then
\begin{align}
&T_{nm,w_K}^{(K)}(z)=T_{n,w_L}^{(L)}(p(z))/a_m^n,\label{preim-Tn}
\\
&t_{nm}(K,w_K)=t_n(L,w_L)/|a_m^n|,\label{preim-tn}
\end{align}
and if, in addition, $\CAP(L)>0$,
\begin{align}
\W_{nm}(K,w_K)&=\W_n(L,w_L).\label{preim-Wn}
\end{align}
\end{theorem}
\begin{proof}
Suppose by contradiction that the monic polynomial of degree $nm$,
\begin{align}\label{Rpoly}
R(z)=T_{n,w_L}^{(L)}(p(z))/a_m^n
\end{align}
is not the weighted Chebyshev polynomial on $K$ with respect to the weight $w_K$. Then by Kolmogorov's criterion, Theorem~\ref{Kolm-Thm}, there exists a polynomial $q(z)=c_{nm-1}z^{nm-1}+\dots+c_0$ such that
\begin{align}\label{NegKolm}
\max_{z_0\in K_0}\Re\bigr[R(z_0)\ol{q(z_0)}\,\bigl]<0,
\end{align}
where $K_0$ is the set of $w_K$-extremal points of $R(z)$ on $K$.

For each $\zeta$ let $z_1,\dots,z_{m}$ denote the solutions of $p(z)=\zeta$ and define
\begin{align}\label{AvgS}
S(\zeta)=\sum_{k=1}^m q(z_k) = \sum_{j=0}^{nm-1}c_{j}S_{j}(\zeta),
\end{align}
where $S_j(\zeta):=z_1^j+\dots+z_m^j$. We will show that $S(\zeta)$ is a polynomial of degree at most $n-1$. By Newton's identities for the power sums $S_j(\zeta)$ of the roots of the polynomial $p(z)-\zeta$, we have
\begin{align}
\sum_{j=1}^k a_{m-k+j}S_j(\zeta)-ka_{m-k}&=0, \quad k<m, \label{Newt1}
\\
\sum_{j=k-m+1}^k a_{m-k+j}S_j(\zeta)+(a_{0}-\zeta)S_{k-m}(\zeta)&=0, \quad k\geq m. \label{Newt2}
\end{align}
Since $a_m\neq0$, it follows from \eqref{Newt1} that $S_k(\zeta)$ is independent of $\zeta$ for each $k=1,\dots,m-1$. Then \eqref{Newt2} implies that $S_k(\zeta)$ is a polynomials of degree at most $1$ for each $k=m,\dots,2m-1$. Proceeding by induction, shows that $S_k(\zeta)$ is a polynomial of degree at most $d$ for each $k=dm,\dots,(d+1)m-1$. Thus, by \eqref{AvgS}, $S(\zeta)$ is a polynomial of degree at most $n-1$.

Let $L_0$ be the set of $w_L$-extremal points of $T_{n,w_L}^{(L)}$ on $L$. Then for each $\zeta\in L_0$ the solutions $z_1,\dots,z_m$ of $p(z)=\zeta$ are the $w_K$-extremal points of $R(z)$ on $K$ by \eqref{Rpoly}. Then by \eqref{NegKolm} and \eqref{Rpoly},
\begin{align}
\Re\bigr[T_{n,w_L}^{(L)}(\zeta)\ol{q(z_j)}\,\bigl] = \Re\bigr[R(z_j)\ol{q(z_j)}\,\bigl] <0, \quad j=1,\dots,m.
\end{align}
Summing the inequalities and using \eqref{AvgS} yield $\Re[T_{n,w_L}^{(L)}(\zeta)\ol{S(\zeta)}]<0$, $\zeta\in L_0$. Then since $S\in\bb P_{n-1}$, Kolmogorov's criterion implies that $T_{n,w_L}^{(L)}$ is not the $w_L$-Chebyshev polynomial on $L$ which is a contradiction. Thus, $R(z)$ must be the $w_K$-Chebyshev polynomial on $K$, that is, \eqref{preim-Tn} holds. Equality for the Chebyshev norms \eqref{preim-tn} then follows from \eqref{preim-Tn}. By \cite[Thm.\,5.2.5]{Ransford-1995}, $\CAP(K)=(\CAP(L)/|a_m|)^{1/m}$ hence \eqref{preim-Wn} follows from \eqref{preim-tn} and \eqref{Wfactors}.
\end{proof}

\section{Bounds for Weighted Chebyshev Polynomials}\label{S3}

In this section we derive several inequalities for weighted Chebyshev polynomials. Throughout this section we assume that $K\subset\bb C$ is an infinite compact set and $w:K\to[0,\infty)$ is a bounded weight function nonzero at infinitely many points of $K$. We start with an elementary comparison theorem.


\begin{theorem}
Let $w_1$, $w_2$ be two weights on a compact set $K\subset\bb C$ and suppose that $w_2(z)\neq0$ for all $z\in K$. Then
\begin{align}\label{tKw1-tKw2}
t_n(K,w_1) \leq \sup_{z\in K}\frac{w_1(z)}{w_2(z)}\,t_n(K,w_2), \quad n\in\bb N.
\end{align}
In particular, if $w$ is a weight on $K$ then comparison to the constant weight $w_0=1$ gives the bounds
\begin{align}\label{tKw-tK}
\inf_{z\in K}w(z)\,t_n(K,w_0) \leq t_n(K,w) \leq \sup_{z\in K}w(z)\,t_n(K,w_0), \quad n\in\bb N.
\end{align}
\end{theorem}
\begin{proof}
Let $T_{n,w_j}^{(K)}$, $j=1,2$, be minimizers for the two weighted problems on $K$. Then
\begin{align}
t_n(K,w_1)&=\norm{w_1T_{n,w_1}^{(K)}}_{K}\leq\norm{w_1 T_{n,w_2}^{(K)}}_{K}=\norm{(w_1/w_2)w_2T_{n,w_2}^{(K)}}_{K}\nonumber\\
&\leq\sup_{z\in K}\frac{w_1(z)}{w_2(z)}\,\norm{w_2T_{n,w_2}^{(K)}}_{K}=\sup_{z\in K}\frac{w_1(z)}{w_2(z)}\,t_n(K,w_2).
\end{align}
Applying \eqref{tKw1-tKw2} to $w_1=w$, $w_2=1$ then yields the upper bound in \eqref{tKw-tK}. If $w$ has a zero on $K$ the lower bound in \eqref{tKw-tK} is trivial, otherwise applying \eqref{tKw1-tKw2} to $w_1=1$, $w_2=w$ yield the lower bound in \eqref{tKw-tK}.
\end{proof}

The lower bound in \eqref{tKw-tK} can be improved for a class of Szeg\H{o} weights which includes weights with $\inf_{z\in K}w(z)=0$. Before we turn to a discussion of such weights we recall several concepts from potential theory that can be found in \cite{MFinkelshtein-2006, Landkof-1972, Ransford-1995, SaffTotik-1997, Tsuji-1975}.


In the following we assume that $K\subset\bb C$ has positive logarithmic capacity $\CAP(K)>0$ and denote by $\Om$ the \emph{outer domain} of $K$, that is, the unbounded component of $\ol{\bb C}\bs K$. The \emph{polynomial convex hull} of $K$ is the set $\hat K:=\bb C\bs\Om$ and $O\pd K:=\pd\Om=\pd\hat K$ is the \emph{outer boundary} of $K$.
The assumption of positive capacity guarantees that there exists a unique \emph{equilibrium measure} $\rho_K$ of $K$ supported on $O\pd K$ and a \emph{Green function} $g_K(z)$ for the domain $\Om$, given by (see e.g., \cite[Eq.\,(I.4.8)]{SaffTotik-1997}),
\begin{align}\label{gK}
g_K(z)=-\log\CAP(K)+\int\log|z-\zeta|\,d\rho_K(\zeta).
\end{align}
By the above formula, the Green function $g_K(z)$ extends to the entire complex plane with $g_K(z)\ge0$ for all $z\in\ol{\bb C}$ and $g_K(z)=0$ for all $z\in\bb C\bs\ol{\Om}$ and quasi-every $z\in\pd\Om$, see \cite[Cor.\,I.4.5]{SaffTotik-1997}. The points $z\in\pd\Om$ where $g_K(z)$ is continuous are called \emph{regular}. By \cite[Thms.\,I.4.4,~II.3.5]{SaffTotik-1997}, $z\in\pd\Om$ is a regular point if and only if $g_K(z)=0$. We also recall a similar notion of the Green function for $\Om$ with a logarithmic pole at a finite point $a\in\Om$ which we denote by $g_K(z,a)$. By Section~II.4 and, in particular, equation (4.31) in \cite{SaffTotik-1997}, we have
\begin{align}\label{gKa}
g_K(z,a) = \log\frac{1}{|z-a|}+\int\log|z-\zeta|d\om_a(\zeta)+g_K(a),
\end{align}
where $\om_z$ is the \emph{harmonic measure} for $\Om$.
By \cite[Thm.\,II.4.4]{SaffTotik-1997}, $g_K(\cdot,a)$ given by the above formula extends to the entire complex plane with $g_K(z,a)\ge0$ for all $z\in\ol{\bb C}$ and $g_K(z,a)=0$ for all $z\in\bb C\bs\ol{\Om}$ and all regular points $z\in\pd\Om$.


For a bounded Borel measurable weight $w$ on $K$ we define the exponential Szeg\H{o} integral of $w$ by
\begin{align}\label{S-def}
S(w) := \exp\left[\int\log w(\zeta)\,d\rho_K(\zeta)\right].
\end{align}
We say that a bounded weight $w$ is of \emph{Szeg\H{o} class} if $\log w\in L^1(d\rho_K)$, equivalently, $S(w)>0$. For explicit computations, it is convenient to note that $S(\cdot)$ is multiplicative, that is, $S(f^\alpha g^\beta)=S(f)^\alpha S(g)^\beta$ and for the special weight $w(z)=|z-z_0|$ we have, by \eqref{gK},
\begin{align}\label{S-spcase}
S(|z-z_0|)=e^{g_K(z_0)}\CAP(K).
\end{align}
This, in particular, allows us to evaluate $S(|f|)$ for polynomial and rational functions $f$.

Finally, we recall that $\rho_K=\om_\infty$ by \cite[Thm.\,4.3.14]{Ransford-1995} and that all harmonic measures for $\Om$ are comparable, $c_1(z)\rho_K\le\om_z\le c_2(z)\rho_K$, by \cite[Cor.\,4.3.5]{Ransford-1995}. Then for any $f\in L^1(d\rho_K)$ we have a well defined generalized Poisson integral,
\begin{align}\label{PI}
\PI_\Om(f,z):=\int f(\zeta)\,d\om_z(\zeta), \quad z\in\Om.
\end{align}
For later reference we also note that
\begin{align}\label{PI-logS}
\PI_\Om(\log f;\infty)=\log S(f).
\end{align}


Our first result for Szeg\H{o} class weights is the following weighted analog of the Bernstein--Walsh inequality. For a recent review and an extension of the classical unweighted Bernstein--Walsh inequality see \cite{Sch-2017b}.

\begin{theorem}\label{Thm-BW}
Let $K\subset\bb C$ be a compact set with $\CAP(K)>0$ and $w$ be a Szeg\H{o} class weight on $K$. Then for any polynomial $P_n$ of degree $n$,
\begin{align}\label{BWineq}
|P_n(z)|\le\norm{P_n}_{K,w}\exp[-\PI_\Om(\log w,z)+ng_K(z)], \quad z\in\Om.
\end{align}
\end{theorem}
\begin{proof}
WLOG assume that $P_n$ is monic and let $P_n(\zeta)=\prod_{j=1}^n(\zeta-z_j)$. Then, by \eqref{PI}, \eqref{gKa}, and linearity of $\PI_\Om(\cdot,z)$ we have
\begin{align}
\PI_\Om(\log|P_n|,z)
&=\int\log\bigg(\prod_{j=1}^n|\zeta-z_j|\bigg)\,d\om_z(\zeta)
=\sum_{j=1}^n\int\log|\zeta-z_j|\,d\om_z(\zeta)
\nonumber\\
&=\sum_{j=1}^n\bigg[g_K(z_j,z)-\log\frac{1}{|z-z_j|}-g_K(z) \bigg]
\nonumber\\
&=\log|P_n(z)|+\sum_{j=1}^n g_K(z_j,z)-ng_K(z), \quad z\in\Om,
\end{align}
and since $\log[w|P_n|]\le\log\norm{P_n}_{K,w}$ on $K$ and $\om_z$ is a probability measure,
\begin{align}\label{logBWineq}
\log\norm{P_n}_{K,w}
&\ge\PI_\Om(\log[w|P_n|],z)=\PI_\Om(\log w,z)+\PI_\Om(\log|P_n|,z)
\nonumber\\
&=\PI_\Om(\log w,z)+\log|P_n(z)|+\sum_{j=1}^n g_K(z_j,z)-ng_K(z)
\nonumber\\
&\ge\PI_\Om(\log w,z)+\log|P_n(z)|-ng_K(z), \quad z\in\Om.
\end{align}
Exponentiating and rearranging then yields \eqref{BWineq}.
\end{proof}


Next we derive a Szeg\H{o}-type lower bound for the weighted Widom factors and show that it is sharp even if $K\subset\bb R$.

\begin{theorem}\label{Thm-LB}
Let $K\subset\bb C$ be a compact set with $\CAP(K)>0$ and $w$ be a Szeg\H{o} class weight on $K$. Then for all $n\in\bb N$,
\begin{align}\label{LB-WnSw}
\W_n(K,w) \geq S(w).
\end{align}
Equality in \eqref{LB-WnSw} is attained for some $n$ if and only if there exists a monic polynomial $P_n$ of degree $n$ with all zeros on the set $\{z\in\bb C:g_K(z)=0\}$ and such that $w|P_n|=\norm{P_n}_{K,w}$ $\rho_K$-a.e, in which case $P_n=T_{n,w}^{(K)}$.

If $K\subset\bb R$ or more generally if $\hat K$ has empty interior, then the inequality \eqref{LB-WnSw} is strict for all $n\in\bb N$. Nevertheless, if $K\subset\bb R$ then for each $n\in\bb N$ the lower bound \eqref{LB-WnSw} is sharp in the class of polynomial weights on $K$, that is,
\begin{align}\label{Inf-WnSw}
\inf_{w}\frac{\W_n(K,w)}{S(w)} = 1,
\end{align}
where the infimum is taken over polynomials $w(z)$ positive on $K$.
\end{theorem}
\begin{proof}
To derive \eqref{LB-WnSw} we apply Theorem~\ref{Thm-BW} to the polynomial $P_n=T_{n,w}^{(K)}$. Taking $z\to\infty$ in \eqref{logBWineq}, recalling \eqref{PI-logS}, and noting that $\log|z|-g_K(z)\to\log\CAP(K)$ and therefore $\log|P_n(z)|-ng_K(z)\to n\log\CAP(K)$ we obtain
\begin{align}\label{logLB-WnSw}
\log\norm{P_n}_{K,w}\ge\log S(w|P_n|)
&=\log S(w)+\sum_{j=1}^n g_K(z_j)+n\log\CAP(K)
\nonumber\\
&\ge\log S(w)+n\log\CAP(K),
\end{align}
where $z_1,\dots,z_n$ are the zeros of $P_n$. Exponentiation then gives \eqref{LB-WnSw}.

The equality $\norm{P_n}_{K,w}=S(w)\CAP(K)^n$ holds if and only if equality is attained in both inequalities of \eqref{logLB-WnSw}, that is, $g_K(z_j)=0$ for all zeros $z_j$ of $P_n$ and
\begin{align}\label{logSwP}
\log S(w|P_n|)=\int\log[w|P_n|]\,d\rho_K = \log\norm{P_n}_{K,w}.
\end{align}
Since $w|P_n|\le\norm{P_n}_{K,w}$ on $K$, \eqref{logSwP} is equivalent to the maximality of the integrand $\rho_K$-a.e., that is, $w|P_n|=\norm{P_n}_{K,w}$ $\rho_K$-a.e. Finally, if a monic polynomial $P_n$ satisfies $\norm{P_n}_{K,w}=S(w)\CAP(K)^n$, then by \eqref{LB-WnSw} we have
$\norm{T_{n,w}^{(K)}}_{K,w}\ge S(w)\CAP(K)^n=\norm{P_n}_{K,w}$ which, by uniqueness of the Chebyshev polynomial, Corollary~\ref{Cor-Uniq}, implies $P_n=T_{n,w}^{(K)}$.

If $\hat K$ has empty interior, then $\ol{\Om}=\ol{\bb C}$. Since $\supp(\rho_K)\subset\pd\Om$ it follows that the complement of the $\supp(\rho_K)$ is connected and hence $g_K>0$ outside of the support of $\rho_K$. Then we have $\{z\in\bb C:g_K(z)=0\}\subset\supp(\rho_K)$ and hence, by the previous part, $T_{n,w}^{(K)}$ has a zero $z_0$ in the support of $\rho_K$. Then every neighborhood of $z_0$ has positive $\rho_K$ measure and since $w$ is bounded $w|T_{n,w}^{(K)}|<\norm{T_{n,w}^{(K)}}_{K,w}$ on a sufficiently small neighborhood of $z_0$. Thus, by the previous part, equality in \eqref{LB-WnSw} does not occur in this case.

Next we show \eqref{Inf-WnSw}. Fix $n\in\bb N$ and, by shifting $K$ if necessary, assume that $0\in K$ and that $0$ is a regular point of $K$, that is, $g_K(z)\to0$ as $z\to0$. Consider the weights $w_\eps(z)=(z^2+\eps^2)^{-n/2}$. Then using \eqref{S-spcase} we obtain
\begin{align}
S(w_\eps)&=S\big(|z+i\eps|^{-n/2}|z-i\eps|^{-n/2}\big)
=S\big(|z+i\eps|\big)^{-n/2}S\big(|z-i\eps|\big)^{-n/2}
\nonumber\\
&=e^{-\frac{n}{2}[g_K(i\eps)+g_K(-i\eps)]}\CAP(K)^{-n}
\to\CAP(K)^{-n} \;\text{ as }\; \eps\to0.
\end{align}
Since
\begin{align}
t_n(K,w_\eps) \leq \norm{x^n}_{K,w_\eps}=\sup_{x\in K}|w_\eps x^n| = \sup_{x\in K} \left|\frac{x^2}{x^2+\eps^2}\right|^{n/2} \leq 1,
\end{align}
we have $\W_n(K,w_\eps) \leq \CAP(K)^{-n}$ and hence
\begin{align}\label{Inf-2}
\limsup_{\eps\to0}\W_n(K,w_\eps)/S(w_\eps)\le1.
\end{align}
Next, fix $\eps\in(0,1)$ and approximate $w_\eps$ by polynomials $w_j$ in the uniform norm on $K$. Since $w_\eps\ge c>0$ on $K$ we may assume $w_j>0$ on $K$ for each $j$ and hence we have
\begin{align}
\Big\|1-\frac{w_j}{w_\eps}\Big\|_K \le \frac1c\|w_\eps-w_j\|_K\to0
\;\text{ as }\; j\to\infty.
\end{align}
Then, $\log(w_j/w_\eps)\to0$ uniformly on $K$ hence
\begin{align}\label{Swj-Swe}
\frac{S(w_j)}{S(w_\eps)} = \exp\left[\int\log\frac{w_j(x)}{w_\eps(x)}d\rho_K\right] \to 1 \;\text{ as }\; j\to\infty.
\end{align}
By \eqref{tKw1-tKw2}, we have
\begin{align}
\limsup_{j\to\infty}t_n(K,w_j)\leq
\limsup_{j\to\infty}\|w_j/w_\eps\|_K t_n(K,w_\eps)=t_n(K,w_\eps),
\end{align}
so $\limsup_{j\to\infty}\W_n(K,w_j)\le\W_n(K,w_\eps)$. Combined with \eqref{Swj-Swe} this yields
\begin{align}\label{Inf-3}
\limsup_{j\to\infty}\frac{\W_n(K,w_j)}{S(w_j)}\leq\frac{\W_n(K,w_\eps)}{S(w_\eps)}.
\end{align}
Finally, \eqref{Inf-WnSw} follows from \eqref{LB-WnSw}, \eqref{Inf-2} and \eqref{Inf-3}.
\end{proof}


\begin{remark}
\begin{enumerate}[(i)]
\item A similar lower bound for $L^p$-extremal polynomials, $p<\infty$, has been recently obtained in \cite[Thm.\,2.1]{AlpZin-JMAA2020}. In fact, \eqref{LB-WnSw} can be alternatively derived from that result.

\item For the constant weight $w_0=1$ we have $S(w_0)=1$. In this case it is shown in \cite[Thm.\,1.2]{CSZ-2020} that the lower bound \eqref{LB-WnSw} is saturated, that is, $\W_n(K,w_0)=1$ if and only if the outer boundary of $K$ is a degree $n$ lemniscate (a polynomial pre-image of the unit circle under a degree $n$ polynomial). A similar saturation result for $L^p(d\rho_K)$-extremal polynomials, $p<\infty$, is derived in \cite[Thm.\,4.2]{AlpZin-JAT2020}.

\item In the weighted case, the lower bound \eqref{LB-WnSw} can be saturated on any set $K$ whose polynomial convex hull has nonempty interior. Indeed, if $\hat K$ has nonempty interior, let $P_n$ be a degree $n$ monic polynomial with all zeros in the interior of $\hat K$ and set $w(z)=1/|P_n(z)|$ for all $z\in\pd\hat K$ and $w(z)=0$ everywhere else on $K$. Then since $\supp(\rho_K)\subset\pd\hat K$, it follows from Theorem~\ref{Thm-LB} that equality is attained in \eqref{LB-WnSw} in this case. In addition, Theorem~\ref{Thm-LB} shows that nonempty interior of $\hat K$ is a necessary condition for the equality in \eqref{LB-WnSw}.

\item For the constant weight $w_0=1$, it is known \cite{Sch-2008b} that the lower bound doubles for subsets of the real line, that is, $\W_n(K,w_0)\geq2$ if $K\subset\bb R$. Theorem~\ref{Thm-LB} shows that no doubling of the lower bound occurs in the class of weighted Chebyshev polynomials on subsets of the real line.

\item However, there are nonconstant weights $w$ on $K\subset\bb R$ for which the lower bound \eqref{LB-WnSw} doubles. For example, if $K\subset\bb R$ is a regular compact set and $w(z)=|P_d(z)|$ for some monic polynomial $P_d(z)$ of degree $d$ with all zeros on $K$, then
\begin{align}
t_n(K,w)=\norm{P_dT_{n,w}^{(K)}}_{K}\geq\norm{T_{n+d,w_0}^{(K)}}_{K}=t_{n+d}(K,w_0),
\end{align}
and hence, by the doubled lower bound for the constant weight $w_0=1$ and \eqref{S-spcase}, we get
\begin{align}
\W_n(K,w)\ge\W_{n+d}(K,w_0)\CAP(K)^d\ge2\CAP(K)^d=2S(w).
\end{align}
    We pose it as an open problem to find a characterization of weights $w$ on $K=[-1,1]$ (and more generally on $K\subset\bb R$) for which the doubled lower bound $\W_n(K,w)\ge2S(w)$ holds.
\end{enumerate}
\end{remark}


\bibliographystyle{amsplain}
\bibliography{ChebRefs}

\end{document}